\documentclass[12pt,draftcls,onecolumn]{IEEEtran}
\usepackage{amsfonts}
\usepackage{amsmath,amssymb}
\usepackage{graphicx}
\usepackage{caption2}
\usepackage{epsfig}
\usepackage[dvips]{color}

\def\dref#1{(\ref{#1})}
\def\rm{\mathrm}
\newtheorem{theorem}{Theorem}
\newtheorem{lemma}{Lemma}

\newtheorem{remark}{Remark}

\begin{document}

\title{Designing Fully Distributed Consensus Protocols
for Linear Multi-agent Systems with Directed Graphs}
%
%
\author{Zhongkui~Li,~\IEEEmembership{Member,~IEEE}, Guanghui~Wen,~\IEEEmembership{Member,~IEEE}, Zhisheng~Duan,
Wei~Ren,~\IEEEmembership{Senior Member,~IEEE}%
\thanks{This work was supported by the National Natural Science Foundation
of China under grants 61104153, 61473005, 11332001, 61225013, 61304168,
a Foundation for the Author of National Excellent Doctoral Dissertation of PR China,
the Natural Science
Foundation of Jiangsu Province of China under grant BK20130595, and
National Science Foundation under grant ECCS-1307678.}
\thanks{Z. Li and Z. Duan are with the State Key Laboratory for Turbulence
and Complex Systems,
Department of Mechanics and Engineering Science, College
of Engineering, Peking University, Beijing 100871, China
(E-mail: zhongkli@pku.edu.cn; duanzs@pku.edu.cn).}
\thanks{G. Wen is with the
Department of Mathematics, Southeast University, Nanjing 210096, China (E-mail: wenguanghui@gmail.com).}
\thanks{W. Ren is with the Department of Electrical Engineering, University of California, Riverside, CA, 92521,
USA (E-mail: ren@ee.ucr.edu).}}

\maketitle

\begin{abstract}
This technical note addresses the distributed
consensus protocol design problem for multi-agent systems
with general linear dynamics and directed
communication graphs.
Existing works usually design consensus protocols
using the smallest real part of the nonzero eigenvalues
of the Laplacian matrix associated with
the communication graph, which however is global information.
In this technical note, based on only the agent dynamics
and the relative states of neighboring agents,
a distributed adaptive consensus protocol is designed
to achieve leader-follower consensus in the presence of a leader with a zero input
for
any communication graph containing a directed spanning
tree with the leader as the root node.
The proposed adaptive
protocol is independent of any global information of
the communication graph and thereby
is fully distributed.
Extensions
to the case with multiple leaders are further studied.
\end{abstract}

\begin{keywords}
Multi-agent system, cooperative control,
consensus, distributed control,
adaptive control.
\end{keywords}

\section{Introduction}

Consensus of multi-agent systems has
been an emerging research topic
in the systems and control community
in recent years.
Due to its potential applications in several areas such as spacecraft
formation flying, sensor networks, and cooperative surveillance
\cite{ren2007information},
the consensus control problem has been addressed by
many researchers from various perspectives;
see \cite{ren2007information,hong2008distributed,olfati-saber2004consensus,antonelli2013interconnected,Guo2013consensus}
and the references therein.
Existing consensus algorithms can be roughly categorized into two classes,
namely, consensus without a leader (i.e., leaderless consensus) and
consensus with a leader which is also called leader-follower
consensus or distributed tracking.
For the consensus control problem,
a key task is
to design appropriate
distributed controllers
which are usually called
consensus protocols.
Due to the spatial distribution of the agents and
limited sensing capability of sensors,
implementable consensus protocols for multi-agent
systems should be distributed, depending
on only the local state or output
information of each agent
and its neighbors.

In this technical note, we consider the distributed consensus protocol design problem for multi-agent systems
with general continuous-time linear dynamics. Previous works along
this line include
\cite{li2010consensus,li2011dynamic,tuna2009conditions,seo2009consensus,zhang2011optimal,ma2010necessary},
where different static and dynamic consensus protocols have been proposed.
One common feature in the aforementioned works
is that the design of the consensus protocols requires
the knowledge of some eigenvalue information of the Laplacian matrix associated with the
communication graph (specifically, the smallest nonzero eigenvalue
of the Laplacian matrix for undirected graphs and the smallest real part
of the nonzero eigenvalues of the Laplacian matrix for directed graphs).
As pointed in \cite{tuna2009conditions},
for the case where the agents are not neutrally stable,
e.g., double integrators,
the design of the consensus
protocols generally depends on the smallest real part
of the nonzero eigenvalues of the Laplacian matrix.
However, it is worth mentioning
that
the smallest real part
of the nonzero eigenvalues of the Laplacian matrix is global information in the sense that
each agent has to know the entire communication graph $\mathcal {G}$ to compute it.
Therefore, the consensus protocols given in the aforementioned papers
cannot be designed by each agent in a fully distributed fashion, i.e.,
using only the local information of its own and neighbors.
To overcome this limitation,
distributed adaptive consensus protocols
are proposed in \cite{li2012adaptiveauto,li2011adaptive}.
Similar adaptive schemes are presented to achieve second-order
consensus with nonlinear dynamics in \cite{su2011adaptive,yu2013distributed}. Note that the protocols in
\cite{li2012adaptiveauto,li2011adaptive,su2011adaptive,yu2013distributed}
are applicable to only undirected communication graphs or leader-follower
graphs where the subgraphs
among the followers are undirected.
How to design fully distributed adaptive
consensus protocols for the case
with general directed graphs is quite challenging
and to the best of our knowledge is still open.
The main difficulty lies
in that the Laplacian matrices
of directed graphs are generally asymmetric, which
renders the construction
of adaptive consensus
protocol and the selection
of appropriate Lyapunov function
far from being easy.

In this technical note, we intend to design fully distributed
consensus protocols for general linear
multi-agent systems with a leader of a zero input and a directed communication
graph. Based on the relative states of neighboring agents,
a distributed adaptive consensus protocol is constructed.
It is shown via a novel Lyapunov function that
the proposed adaptive protocol can achieve leader-follower consensus
for any communication graph containing a directed spanning
tree with the leader as the root node.
The adaptive protocol proposed
in this technical note, independent of any global information of
the communication graph, relies
on only the agent dynamics
and the relative state information
and thereby is fully distributed.
Compared to the distributed adaptive protocols in \cite{li2012adaptiveauto,li2011adaptive}
for undirected graphs, a distinct feature of
the adaptive protocol in this technical note is that
monotonically increasing functions,
inspired by the changing supply function notion
in \cite{sontag1995changing},
are introduced to provide extra freedom for design.
As an extension, we consider the case where
there exist multiple leaders with zero inputs. In this case, it is shown
that the proposed adaptive protocol can solve
the containment control problem,
i.e., the states of
the followers are to be driven into the convex hull
spanned by the states of the leaders, if
for each follower,
there exists at least one leader that has a directed path to that follower.
A sufficient condition for the existence of
the adaptive protocol in this technical note is that each agent
is stabilizable.

The rest of this technical note is organized as follows.
Mathematical preliminaries required in this paper are summarized in Section II.
The problem is formulated
and the motivation is stated in Section III.
Distributed adaptive consensus protocols are designed
in Section IV for general directed leader-follower graphs.
Extensions to the case with multiple leaders
are studied in Section V. Simulation examples are presented
for illustration in Section VI. Conclusions are
drawn in Section VII.

\section{Mathematical Preliminaries}

Throughout this technical note, the following
notations and definitions will be used:
$\mathbf{R}^{n\times m}$ and $\mathbf{C}^{n\times m}$ denote the sets of  $n\times m$
real and complex matrices, respectively.
$I_N$ represents the identity matrix of
dimension $N$.
Denote by $\mathbf{1}$ a column vector
with all entries equal to one. ${\rm{diag}}(A_1,\cdots,A_n)$
represents a block-diagonal matrix with matrices $A_i,i=1,\cdots,n,$
on its diagonal. For real symmetric matrices $X$
and $Y$, $X>(\geq)Y$ means that $X-Y$ is positive (semi-)definite.
For a vector $x$, $x >(\geq)0$ means that every entry of $x$ is positive (nonnegative).
$A\otimes B$ denotes the Kronecker product of matrices $A$ and $B$.
${\rm{Re}}(\alpha)$ represents the real part
of $\alpha\in\mathbf{C}$. 
A matrix $A =[a_{ij} ]\in \mathbf{R}^{n\times n}$ is
called a nonsingular $M$-matrix, if $a_{ij} < 0$, $ \forall i \neq j$,
and all eigenvalues of $A$ have positive real parts.

A directed graph $\mathcal {G}$ is a pair $(\mathcal {V}, \mathcal
{E})$, where $\mathcal {V}=\{v_1,\cdots,v_N\}$ is a nonempty finite
set of nodes and $\mathcal {E}\subseteq\mathcal {V}\times\mathcal
{V}$ is a set of edges, in which an edge is represented by an
ordered pair of distinct nodes. For an edge $(v_i,v_j)$, node $v_i$
is called the parent node, node $v_j$ the child node, and $v_i$ is a
neighbor of $v_j$. A graph with the property that
$(v_i,v_j)\in\mathcal {E}$ implies $(v_j, v_i)\in\mathcal {E}$ for
any $v_i,v_j\in\mathcal {V}$ is said to be undirected. A path from
node $v_{i_1}$ to node $v_{i_l}$ is a sequence of ordered edges of
the form $(v_{i_k}, v_{i_{k+1}})$, $k=1,\cdots,l-1$. 
A directed graph contains a directed spanning tree if
there exists a node called the root, which has no parent node, such
that the node has directed paths to all other nodes in the graph.
%

The adjacency matrix $\mathcal {A}=[a_{ij}]\in\mathbf{R}^{N\times
N}$ associated with the directed graph $\mathcal {G}$ is defined by
$a_{ii}=0$, $a_{ij}$ is a positive value if $(v_j,v_i)\in\mathcal {E}$ and $a_{ij}=0$
otherwise. Note that $a_{ij}$ denotes the weight for the edge $(v_j,v_i)\in\mathcal {E}$.
If the weights are not relevant, then $a_{ij}$ is set equal to 1
if $(v_j,v_i)\in\mathcal {E}$. The Laplacian matrix $\mathcal {L}=[\mathcal
{L}_{ij}]\in\mathbf{R}^{N\times N}$ is defined as $\mathcal
{L}_{ii}=\sum_{j\neq i}a_{ij}$ and $\mathcal {L}_{ij}=-a_{ij}$,
$i\neq j$. 

\begin{lemma}[\cite{ren2005consensus}]
Zero is an eigenvalue of $\mathcal {L}$ with $\mathbf{1}$ as a
right eigenvector and all nonzero eigenvalues have positive real
parts. Furthermore, zero is a simple eigenvalue of $\mathcal {L}$ if
and only if $\mathcal {G}$ has a directed spanning tree.
\end{lemma}
%
%


\begin{lemma}[Young's Inequality, \cite{bernstein2009matrix}]
If $a$ and $b$ are nonnegative real numbers and $p$ and $q$ are positive real numbers such that $1/p + 1/q = 1$, then
$ab\leq\frac{a^p}{p}+\frac{b^q}{q}$.
\end{lemma}

\section{Problem Statement and Motivations}

Consider a group of $N+1$ identical agents with general
linear dynamics, consisting of
$N$ followers and a leader.
The dynamics of the $i$-th agent
are described by
\begin{equation}\label{1c}
\begin{aligned}
    \dot{x}_i &=Ax_i+Bu_i,
\quad i=0,\cdots,N,
\end{aligned}
\end{equation}
where $x_i\in\mathbf{R}^n$ is the state,
$u_i\in\mathbf{R}^{p}$ is the control input,
and $A$ and $B$ are constant matrices with
compatible dimensions.

Without loss of generality, let the agent in \dref{1c} indexed by 0
be the leader (which receives no information
from any follower) and the agents indexed by $1,\cdots,N$, be the
followers. It is assumed that
the leader's control input is zero, i.e., $u_0=0$.
The communication graph $\mathcal {G}$ among the $N+1$ agents is assumed to satisfy
the following assumption.

{\it Assumption 1:}
The graph $\mathcal {G}$ contains a directed spanning tree with the leader as the root node
\footnote{Equivalently, the leader has directed paths to all followers.}.

Denote by $\mathcal {L}$ the Laplacian matrix associated with
$\mathcal {G}$. Because the node indexed by 0
is the leader which has no neighbors, $\mathcal {L}$
can be partitioned as
\begin{equation}\label{lapc}
\mathcal {L}=\begin{bmatrix} 0 & 0_{1\times N} \\
\mathcal {L}_2 & \mathcal {L}_1\end{bmatrix},
\end{equation}
where $\mathcal {L}_2\in\mathbf{R}^{N\times 1}$ and $\mathcal
{L}_1\in\mathbf{R}^{N\times N}$.
Since $\mathcal {G}$ satisfies Assumption 1,
it follows from Lemma 1 that all eigenvalues of $\mathcal {L}_1$ have positive real parts.
It then can be verified that $\mathcal {L}_1$ is a nonsingular $M$-matrix and
is diagonally dominant.

The intention of this technical note is to solve the leader-follower
consensus problem for the agent in \dref{1c}, i.e., to design
distributed consensus protocols under which the states of the $N$ followers
converge to the state of the leader in the sense of
$\lim_{t\rightarrow \infty}\|x_i(t)- x_0(t)\|=0$,
$\forall\,i=1,\cdots,N.$

Several consensus protocols have been
proposed to reach leader-follower
consensus for the agents in \dref{1c}, e.g., in
\cite{li2010consensus,li2011dynamic,
seo2009consensus,tuna2008lqr,tuna2009conditions,zhang2011optimal}.
A static consensus protocol based on the relative states
between neighboring agents is given in \cite{li2010consensus,zhang2011optimal} as
\begin{equation}\label{clc1}
\begin{aligned}
u_i =c\tilde{K}\sum_{j=1}^Na_{ij}(x_i-x_j),\quad i=1,\cdots,N,
\end{aligned}
\end{equation}
where $c>0$ is the common coupling weight among neighboring agents,
$\tilde{K}\in\mathbf{R}^{p\times n}$ is the feedback gain matrix, and
$a_{ij}$ is $(i,j)$-th entry of the adjacency matrix $\mathcal {A}$ associated with
$\mathcal {G}$.

\begin{lemma}[\cite{li2010consensus,zhang2011optimal}]
For the communication graph $\mathcal
{G}$ satisfying Assumption 1,
the $N$ agents described by \dref{1c} reach leader-follower consensus
under the protocol \dref{clc1} with $\tilde{K}=-B^TP^{-1}$ and
$c\geq 1/\underset{i=1,\cdot\cdot,N}{\min}\mathrm{Re}(\lambda_i)$,
where $\lambda_i$, $i=1,\cdots,N$, are the nonzero eigenvalues of
$\mathcal {L}_1$ and $P>0$ is a
solution to the following linear matrix inequality (LMI):
\begin{equation}\label{alg1}
AP+PA^T-2BB^T<0.
\end{equation}
\end{lemma}

As shown in the above lemma, in order to reach consensus,
the coupling weight $c$ should be not
less than the inverse of $\underset{i=1,\cdot\cdot,N}{\min}\mathrm{Re}(\lambda_i)$,
that is, the smallest real part of
the eigenvalues of $\mathcal {L}_1$. Actually
it is pointed out
in \cite{tuna2009conditions,tuna2008lqr} that
for the case where the agents in \dref{1c}
are critically unstable,
e.g., double integrators,
the design of the consensus
protocol
generally requires the knowledge
of $\underset{i=1,\cdot\cdot,N}{\min}\mathrm{Re}(\lambda_i)$.
However, it is worth mentioning
that $\underset{i=1,\cdot\cdot,N}{\min}\mathrm{Re}(\lambda_i)$ is global information in the sense that
each follower has to know the entire communication graph $\mathcal {G}$ to compute it.
Therefore, the consensus protocols given in Lemma 3
cannot be designed by each agent in a fully distributed fashion, i.e.,
using only the local information of its own and neighbors.
This limitation motivates us
to design some fully distributed consensus protocols for the agents
in \dref{1c} whose directed communication graph satisfies Assumption 1.

\section{Distributed Adaptive Consensus Protocol Design}

In this section, we consider the case where each agent has access
to the relative states of its neighbors with respect
to itself.
Based on the relative states of neighboring agents,
we propose the following distributed adaptive consensus protocol
with time-varying coupling weights:
\begin{equation}\label{ssda}
\begin{aligned}
u_i &=c_i\rho_i(\xi_i^TP^{-1}\xi_i)K\xi_i,\\
\dot{c}_i &=\xi_i^T\Gamma\xi_i,\quad i=1,\cdots,N,
\end{aligned}
\end{equation}
where $\xi_i\triangleq\sum_{j=0}^Na_{ij}(x_i-x_j)$,
$c_i(t)$ denotes the
time-varying coupling weight associated with the $i$-th follower
with $c_i(0)\geq1$,
$P>0$ is a
solution to the LMI \dref{alg1},
$K\in\mathbf{R}^{p\times n}$ and $\Gamma\in\mathbf{R}^{n\times n}$
are the feedback gain matrices to be designed,
$\rho_i(\cdot)$ are smooth
and monotonically increasing functions
to be determined later which satisfies that
$\rho_i(s)\geq1$ for $s>0$, and the rest of variables
are defined as in \dref{clc1}.

Let $\xi=[\xi_1^T,\cdots,\xi_N^T]^T$. Then,
\begin{equation}\label{conerr}
\xi=(\mathcal {L}_1\otimes I_n)\begin{bmatrix} x_1-x_0\\\vdots\\x_N-x_0\end{bmatrix},
\end{equation}
where $\mathcal {L}_1$
is defined in \dref{lapc}. Because $\mathcal {L}_1$
is nonsingular for $\mathcal
{G}$ satisfying Assumption 1, it is easy
to see that the leader-follower consensus problem
is solved if and only if $\xi$ asymptotically
converges to zero.
Hereafter,
we refer to $\xi$ as the consensus error.
In light of \dref{1c} and \dref{ssda},
it is not difficult to obtain that
$\xi$ and $c_i$ satisfy the following dynamics:
\begin{equation}\label{netss1}
\begin{aligned}
\dot{\xi}
&= [I_N\otimes A+\mathcal {L}_1\widehat{C}\hat{\rho}(\xi)\otimes BK]\xi,\\
\dot{c}_i &=\xi_i^T\Gamma\xi_i,
\end{aligned}
\end{equation}
where
$\hat{\rho}(\xi)\triangleq{\rm{diag}}(\rho_1(\xi_1^TP^{-1}\xi_1),\cdots,\rho_N(\xi_N^TP^{-1}\xi_N))$ and
$\widehat{C}\triangleq{\rm{diag}}(c_1,\cdots,c_N)$.

Before moving on to present the main result of this section,
we first introduce a property of the nonsingular $M$-matrix $\mathcal {L}_1$.

\begin{lemma}\label{ch4lemdir1}
There exists a positive diagonal matrix $G$ such that $G\mathcal {L}_1+\mathcal {L}_1^TG>0$.
One such $G$ is given by ${\mathrm{diag}}(q_1,\cdots,q_{N})$, where $q =[q_1,\cdots,q_{N}]^T=(\mathcal {L}_1^T)^{-1}{\bf 1}$.
\end{lemma}

\begin{proof}
The first assertion is well known; see Theorem 4.25 in \cite{qu2009cooperative} or Theorem 2.3 in \cite{berman1994nonnegative}.
The second assertion is shown in the following. Note that the specific form of $G$ given here is different from that in \cite{qu2009cooperative,das2010distributed,zhang2012lyapunov}.

Since $\mathcal {L}_1$ is a nonsingular $M$-matrix, it follows from Theorem 4.25 in \cite{qu2009cooperative}
that $(\mathcal {L}_1^T)^{-1}$ exists, is nonnegative, and thereby cannot have a zero row. Then, it is easy to
verify that
$q>0$ and hence $G\mathcal {L}_1{\bf 1}\geq 0$.
By noting that $\mathcal {L}_1^TG {\bf 1}=\mathcal {L}_1^T q={\bf 1}$, we can conclude that $(G\mathcal {L}_1+\mathcal {L}_1^TG){\bf 1}>0$, implying that $G\mathcal {L}_1+\mathcal {L}_1^TG$ is strictly diagonally dominant. Since the diagonal entries of $G\mathcal {L}_1+\mathcal {L}_1^TG$ are positive, it then follows from Gershgorin's disc theorem \cite{bernstein2009matrix} that every eigenvalue of $G\mathcal {L}_1+\mathcal {L}_1^TG$ is positive, implying that $G\mathcal {L}_1+\mathcal {L}_1^TG>0$.
\end{proof}

The following result provides a sufficient condition
to design the adaptive consensus protocol \dref{ssda}.

\begin{theorem}
Suppose that the communication graph $\mathcal
{G}$ satisfies Assumption 1.
Then, the leader-follower consensus problem of the agents
in \dref{1c} is solved by the adaptive protocol \dref{ssda}
with $K=-B^TP^{-1}$, $\Gamma=P^{-1}BB^TP^{-1}$,
and $\rho_i(\xi_i^TP^{-1}\xi_i)=(1+\xi_i^TP^{-1}\xi_i)^3$,
where $P>0$ is a
solution to the LMI \dref{alg1}.
Moreover, each coupling weight
$c_{i}$ converges to some finite steady-state value.
\end{theorem}

\begin{proof}
Consider the following Lyapunov function candidate:
\begin{equation}\label{lyas1}
V_1=\sum_{i=1}^N\frac{c_iq_i}{2}\int_0^{\xi_i^TP^{-1}\xi_i}\rho_i(s)ds+\frac{\hat{\lambda}_0}{24}\sum_{i=1}^N\tilde{c}_i^2,
\end{equation}
where $\tilde{c}_i=c_i-\alpha$,
$\alpha$ is a positive scalar to be determined later, $\hat{\lambda}_0$
denotes the smallest eigenvalue of $G\mathcal {L}_1+\mathcal {L}_1^TG$,
and $G\triangleq{\rm{diag}}(q_1,\cdots,q_{N})$ is chosen as in Lemma 4
such that $G\mathcal {L}_1+\mathcal {L}_1^TG>0$.
Because $c_i(0)\geq1$, it follows from the second equation in \dref{netss1}
that $c_i(t)\geq1$ for $t>0$. Furthermore, by noting that
$\rho_i(\cdot)$ are smooth and monotonically increasing functions
satisfying $\rho_i(s)\geq1$ for $s>0$, it is not difficult to see
that $V_1$ is positive definite with respect to $\xi_i$ and $\tilde{c}_i$, $i=1,\cdots,N$.

The time derivative of $V_1$
along the trajectory of \dref{netss1} is given by
\begin{equation}\label{lyas2}
\begin{aligned}
\dot{V}_1 &=
\sum_{i=1}^N c_iq_i \rho_i(\xi_i^TP^{-1}\xi_i)\xi_i^TP^{-1}\dot{\xi}_i\\
&\quad+
\sum_{i=1}^N\frac{\dot{c}_iq_i}{2}\int_0^{\xi_i^TP^{-1}\xi_i}\rho_i(s)ds\\
&\quad
+\frac{\hat{\lambda}_0}{12}\sum_{i=1}^N(c_i-\alpha)\xi_i^TP^{-1}BB^TP^{-1}\xi_i.
\end{aligned}
\end{equation}
In the sequel, for conciseness we shall use $\hat{\rho}$
and $\rho_i$ instead of $\hat{\rho}(\xi)$
and $\rho_i(\xi_i^TP^{-1}\xi_i)$, respectively,
whenever without causing any confusion.

By using \dref{netss1} and after some mathematical manipulations,
we can get that
\begin{equation}\label{lyas3}
\begin{aligned}
\sum_{i=1}^N c_i q_i \rho_i\xi_i^TP^{-1}\dot{\xi}_i&=\xi^T(\widehat{C}\hat{\rho} G \otimes P^{-1})\dot{\xi}\\
&=\frac{1}{2}\xi^T[\widehat{C}\hat{\rho} G \otimes (P^{-1}A+A^TP^{-1})\\
&\quad-
\widehat{C}\hat{\rho} (G\mathcal {L}_1+\mathcal {L}_1^TG)\widehat{C}\hat{\rho} \otimes P^{-1}BB^TP^{-1}]\xi\\
&\leq\frac{1}{2}\xi^T[\widehat{C}\hat{\rho} G \otimes (P^{-1}A+A^TP^{-1})\\
&\quad-\hat{\lambda}_0
\widehat{C}^2\hat{\rho}^2 \otimes P^{-1}BB^TP^{-1}]\xi,
\end{aligned}
\end{equation}
where we have used the fact that $G\mathcal {L}_1+\mathcal {L}_1^TG\geq\hat{\lambda}_0 I$ to get the first inequality.

Because $\rho_i$ are monotonically
increasing and satisfy $\rho_i(s)\geq1$ for $s>0$, it follows that
\begin{equation}\label{lyas4}
\begin{aligned}
\sum_{i=1}^N \dot{c}_i q_i \int_0^{\xi_i^TP^{-1}\xi_i}\rho_i(s)ds
&\leq \sum_{i=1}^N \dot{c}_i q_i \rho_i\xi_i^TP^{-1}\xi_i\\
&\leq \sum_{i=1}^N\frac{\dot{c}_iq_i^3}{3\hat{\lambda}_0^2}+\sum_{i=1}^N\frac{2}{3}\hat{\lambda}_0\dot{c}_i\rho_i^{\frac{3}{2}}(\xi_i^TP^{-1}\xi_i)^{\frac{3}{2}}\\
&\leq \sum_{i=1}^N\frac{\dot{c}_iq_i^3}{3\hat{\lambda}_0^2}+\sum_{i=1}^N\frac{2}{3}\hat{\lambda}_0\dot{c}_i\rho_i^{\frac{3}{2}}(1+\xi_i^TP^{-1}\xi_i)^{\frac{3}{2}}\\
&= \sum_{i=1}^N(\frac{q_i^3}{3\hat{\lambda}_0^2}+\frac{2}{3}\hat{\lambda}_0\rho_i^2)\xi_i^TP^{-1}BB^TP^{-1}\xi_i,
\end{aligned}
\end{equation}
where we have used the mean value theorem for integrals
to get the first inequality and used Lemma 2 to get the second inequality.

Substituting \dref{lyas3} and \dref{lyas4} into \dref{lyas2} yields
\begin{equation}\label{lyas5}
\begin{aligned}
\dot{V}_1 &\leq
\frac{1}{2}\xi^T[\widehat{C}\hat{\rho} G \otimes (P^{-1}A+A^TP^{-1})]\xi\\
&\quad-\sum_{i=1}^N
[\hat{\lambda}_0(\frac{1}{2}c_i^2\rho_i^2-\frac{1}{12}c_i-\frac{1}{3}\rho_i^2)
\\
&\quad+\frac{1}{12}(\hat{\lambda}_0\alpha-
\frac{2 q_i^3}{\hat{\lambda}_0^2})]\xi^T_iP^{-1}BB^TP^{-1}\xi_i.
\end{aligned}
\end{equation}
Choose $\alpha\geq\hat{\alpha}+\max_{i=1,\cdots,N}\frac{2 q_i^3}{\hat{\lambda}_0^3}$,
where $\hat{\alpha}>0$ will be determined later. Then, by noting that $\rho_i\geq 1$
and $c_i\geq1$, $i=1,\cdots,N$,
it follows from \dref{lyas5}
that
\begin{equation}\label{lyas6}
\begin{aligned}
\dot{V}_1 &\leq
\frac{1}{2}\xi^T[\widehat{C}\hat{\rho} G \otimes (P^{-1}A+A^TP^{-1})]\xi\\
&\quad-\frac{\hat{\lambda}_0}{12}\sum_{i=1}^N
(c_i^2\rho_i^2+\hat{\alpha}) \xi^T_iP^{-1}BB^TP^{-1}\xi_i\\
&\leq
\frac{1}{2}\xi^T[\widehat{C}\hat{\rho} G \otimes (P^{-1}A+A^TP^{-1})\\
&\quad-\frac{1}{3}\sqrt{\hat{\alpha}}\hat{\lambda}_0
\widehat{C}\hat{\rho} \otimes P^{-1}BB^TP^{-1}]\xi.
\end{aligned}
\end{equation}
Let $\tilde{\xi}=(\sqrt{\widehat{C}\hat{\rho} G}\otimes I)\xi$ and choose $\hat{\alpha}$
to be sufficiently large such that $\sqrt{\hat{\alpha}}\hat{\lambda}_0G^{-1}
\geq 6I$. Then, we can get from \dref{lyas6} that
\begin{equation}\label{lyas7}
\begin{aligned}
\dot{V}_1 &\leq
\frac{1}{2}\tilde{\xi}^T[I_N \otimes (P^{-1}A+A^TP^{-1}-2 P^{-1}BB^TP^{-1})]\tilde{\xi}\\
&\leq 0,
\end{aligned}
\end{equation}
where to get the last inequality we have used
the assertion that
$P^{-1}A+A^TP^{-1}-2 P^{-1}BB^TP^{-1}<0$, which follows readily from \dref{alg1}.

Since $\dot{V}_1\leq 0$, $V_1(t)$ is bounded and so is each $c_{i}$.
By noting $\dot{c}_i\geq 0$, it can be seen from \dref{netss1}
that $c_{i}$ are monotonically increasing.
Then, it follows that each coupling weight $c_{i}$ converges to some finite value.
Note that $\dot{V}_1\equiv0$
implies that $\tilde{\xi}=0$ and thereby $\xi=0$.
Hence, by LaSalle's Invariance principle
\cite{krstic1995nonlinear}, it follows that
the consensus error $\xi$ asymptotically converges
to zero. That is, the leader-follower consensus problem is solved.
\end{proof}

\begin{remark}
As shown in \cite{li2010consensus}, a
necessary and sufficient condition for the existence of a $P>0$ to
the LMI \dref{alg1} is that $(A,B)$ is stabilizable. Therefore, a
sufficient condition for the existence of an adaptive protocol \dref{ssda}
satisfying Theorem 1 is that $(A,B)$ is stabilizable.
The consensus protocol \dref{ssda} can also be designed
by solving the algebraic Ricatti equation: $A^TQ+QA+I-QBB^TQ=0$,
as in \cite{tuna2008lqr,zhang2011optimal}.
In this case, the parameters in \dref{ssda} can be chosen
as $K=-B^TQ$, $\Gamma=QBB^TQ$,
and $\rho_i=(1+\xi_i^TQ\xi_i)^3$.
The solvability of the above Ricatti equation is equal to that of the LMI \dref{alg1}.
\end{remark}

\begin{remark}
In contrast to the consensus protocols in \cite{li2010consensus,li2011dynamic,seo2009consensus,tuna2008lqr,zhang2011optimal}
which require the knowledge of the eigenvalues of the asymmetric Laplacian matrix,
the adaptive protocol \dref{ssda} depends
on only the agent dynamics and the relative states of neighboring
agents, and thereby can be computed and implemented by each agent
in a fully distributed way. Compared to the adaptive protocols
in \cite{li2011adaptive,li2012adaptiveauto} which are applicable
to only undirected graphs or leader-follower graphs
where the subgraph among the followers is undirected,
the adaptive consensus protocol \dref{ssda}
works for general directed leader-follower communication graphs.
\end{remark}

\begin{remark}
In comparison to the adaptive protocols in \cite{li2011adaptive,li2012adaptiveauto},
a distinct feature of \dref{ssda} is that
inspired by the changing supply functions
in \cite{sontag1995changing},
monotonically increasing functions $\rho_i$
are introduced into \dref{ssda}
to provide extra freedom for design.
As the consensus error $\xi$ converges to zero, the functions
$\rho_i$ will converge to 1,
in which case the adaptive protocol \dref{ssda}
will reduce to the adaptive protocols
for undirected graphs in \cite{li2011adaptive,li2012adaptiveauto}.
It is worth mentioning
that the Lyapunov function used in the proof of Theorem 1
is partly motivated by \cite{wang2013leader} which designs
adaptive consensus protocols for first-order multi-agent
systems with uncertainties.
\end{remark}

\section{Extensions to The Case with
Multiple Leaders}

In this section, we extend to consider the case
where there exist more than one leader.
In the presence of multiple leaders, the
containment control problem arises,
where the states of
the followers are to be driven into the convex hull
spanned by those of the leaders \cite{cao2012distributed}.

The dynamics of the $N+1$ agents are still described by \dref{1c}.
Without loss of generality, the agents indexed by $0,\cdots,M$ ($M<N-1$), are the leaders
whose control inputs are assumed to be zero and the agents indexed by
$M+1,\cdots,N$, are the followers. The communication
graph $\mathcal {G}$ among the $N+1$ agents is assumed to satisfy
the following assumption.

{\it Assumption 2:}
For each follower,
there exists at least one leader that has a directed path to that follower.

Clearly, Assumption 2 will reduce to Assumption 1 if only one leader exists.
Under Assumption 2, the Laplacian matrix $\mathcal {L}$ can be written
as
$
\mathcal {L}=\begin{bmatrix}0_{(M+1)\times (M+1)} & 0_{(M+1)\times (N-M)}\\
\widetilde{\mathcal {L}}_2 & \widetilde{\mathcal {L}}_1\\
\end{bmatrix},$
where $\widetilde{\mathcal {L}}_1\in\mathbf{R}^{(N-M)\times (N-M)}$,
$\widetilde{\mathcal {L}}_2\in\mathbf{R}^{(M+1)\times (N-M)}$,
and the following result holds.

\begin{lemma}\cite{cao2012distributed}\label{lems4}
Under Assumption 2, all the eigenvalues of
$\widetilde{\mathcal {L}}_1$ have positive real parts, each entry of $- \widetilde{\mathcal
{L}}_1^{-1}\widetilde{\mathcal {L}}_2$ is nonnegative, and each row of
$-\widetilde{\mathcal {L}}_1^{-1}\widetilde{\mathcal {L}}_2$ has a sum equal to one.
\end{lemma}

In the following, we will investigate if the proposed
adaptive protocol \dref{ssda} can solve the containment
control problem for the case where
$\mathcal {G}$ satisfies
Assumption 2. In this case,
$\zeta\triangleq[\xi_{M+1}^T,\cdots,\xi_{M+1}^T]^T$,
where $\xi_i$ is defined in \dref{ssda},
can be written into
\begin{equation}\label{contain}
\zeta=(\widetilde{\mathcal {L}}_2\otimes I_n)\begin{bmatrix} x_0\\\vdots\\x_M\end{bmatrix}
+(\widetilde{\mathcal {L}}_1\otimes I_n)\begin{bmatrix} x_{M+1}\\\vdots\\x_{N}\end{bmatrix}.
\end{equation}
Clearly, $\zeta=0$ if and only if $\begin{bmatrix} x_{M+1}\\\vdots\\x_{N}\end{bmatrix}=-(\widetilde{\mathcal {L}}_1^{-1}\widetilde{\mathcal {L}}_2
\otimes I_n)\begin{bmatrix} x_0\\\vdots\\x_M\end{bmatrix}$,
which, in light of Lemma \ref{lems4}, implies that
the states of the followers $x_i$, $i=M+1,\cdots,N,$
lie within the convex hull spanned by the states
of the leaders. Thus,
the containment control problem can be reduced to the asymptotical
stability of $\zeta$.
By using \dref{1c}, \dref{ssda}, and \dref{contain},
we can obtain that
$\zeta$ and $c_i$ satisfy
$$
\begin{aligned}
\dot{\zeta}
&= [I_{N-M}\otimes A+\widetilde{\mathcal {L}}_1\widetilde{C}\tilde{\rho}(\zeta)\otimes BK]\zeta,\\
\dot{c}_i &=\xi_i^T\Gamma\xi_i, \quad i=M+1,\cdots,N,
\end{aligned}
$$
where
$\tilde{\rho}(\zeta)\triangleq{\rm{diag}}(\rho_{M+1}(\xi_{M+1}^TP^{-1}\xi_{M+1}),\cdots,\rho_N(\xi_N^TP^{-1}\xi_N))$ and
$\widetilde{C}\triangleq{\rm{diag}}(c_{M+1},\cdots,c_N)$.

The following theorem can be shown by following similar steps in proving Theorem 1.

\begin{theorem}
Supposing that $\mathcal {G}$ satisfies Assumption 2,
the adaptive protocol \dref{ssda}
designed as in Theorem 1 can solve the containment
control problem for the agents in \dref{1c}
and each coupling weight
$c_{i}$ converges to some finite steady-state value.
\end{theorem}

\begin{remark}
Theorem 2 extends Theorem 1 to the case with multiple leaders.
For the case with only one leader, Theorem 2 will
reduce to Theorem 1. The containment control problem of
general linear multi-agent systems was previously discussed
in \cite{li2013distributed}. Contrary
to the static controller in \cite{li2013distributed}
whose design relies on the smallest real part of the eigenvalues
of $\widetilde{\mathcal {L}}_1$,
the adaptive protocol \dref{ssda}
depends on only local information and thus
is fully distributed.
\end{remark}

\section{Simulation Example}

\begin{figure}[htbp]
\centering
\includegraphics[width=0.4\linewidth]{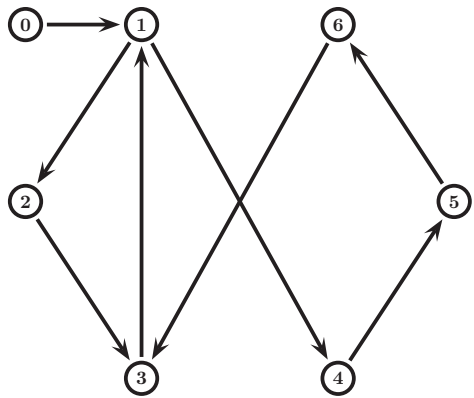}
\caption{The directed communication graph. }
\end{figure}

In this section,
a simulation example is provided for illustration.

Consider a network of third-order integrators,
described by \dref{1c}, with
$$A=\left[\begin{matrix}0 & 1 & 0\\ 0 & 0 & 1 \\ 0 & 0 &0\end{matrix}\right],\quad
B
=\left[\begin{matrix} 0 \\ 0\\1\end{matrix}\right].$$ 
The communication
graph is given as in Fig. 1, where the node indexed by
0 is the leader which is only accessible to the node indexed by 1.
The weights of the graph are randomly chosen within the interval $(0,3)$.
It is easy to verify that the graph in Fig. 1 satisfies
Assumption 1.
%

Solving the LMI \dref{alg1} by using the SeDuMi toolbox \cite{sturm1999using}
gives a solution
$$P =\begin{bmatrix} 3.0861 &  -0.6245  & -0.5186\\
   -0.6245  &  1.1602 &  -0.5573\\
   -0.5186 &  -0.5573  &  0.9850\end{bmatrix}.$$
Thus, the feedback gain matrices in
\dref{ssda} are obtained as
$$\begin{aligned}
K &=-\begin{bmatrix}0.6285  & 1.3525  & 2.1113\end{bmatrix},\\
\Gamma &=\begin{bmatrix}0.3950  &  0.8500 &   1.3269\\
    0.8500 &   1.8292 &   2.8554\\
    1.3269  &  2.8554 &   4.4574\end{bmatrix}.
\end{aligned}$$
To illustrate Theorem 1, let
the initial states $c_i(0)$
and $\rho_i(0)$ be randomly chosen within the interval $[1,3]$.
The consensus errors $x_i-x_0$, $i=1,\cdots,6$, of the double integrators,
under the adaptive protocol \dref{ssda} with
$K$, $\Gamma$, and $\rho_i$ chosen as in Theorem 1,
are depicted in in Fig. 2, which states that leader-follower consensus is indeed
achieved. The coupling weights
$c_{i}$ associated with the followers are drawn in Fig. 3, from which it can be observed that
the coupling weights converge to finite steady-state values.

\begin{figure}[htbp]
\centering
\includegraphics[width=0.7\linewidth,height=0.4\linewidth]{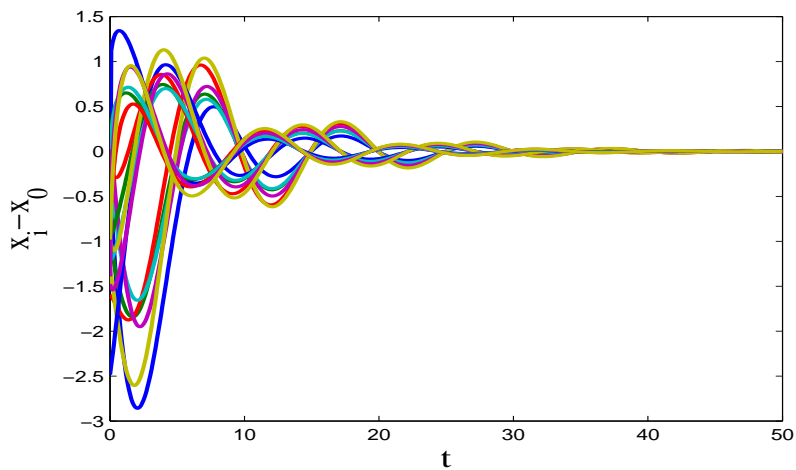}
\caption{$x_i-x_0$, $i=1,\cdots,6$, of third-order integrators under the adaptive protocol \dref{ssda}. }
\end{figure}

\begin{figure}[htbp]
\centering
\includegraphics[width=0.7\linewidth,height=0.4\linewidth]{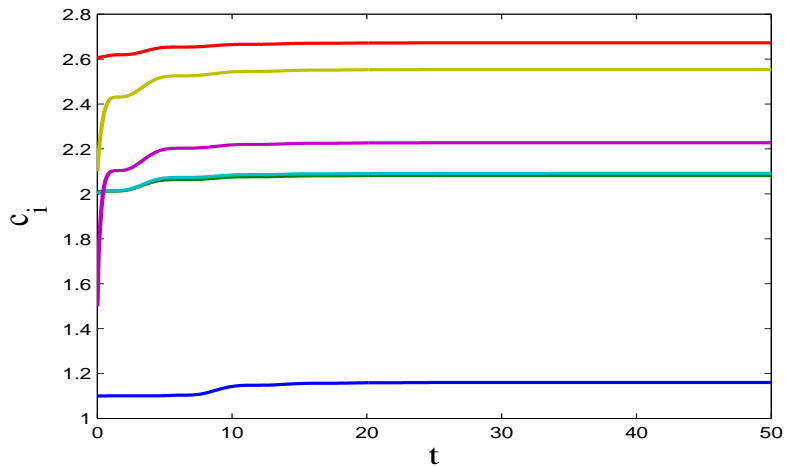}
\caption{The adaptive coupling gains \textcolor{red}{$c_{i}$} in \dref{ssda}. }
\end{figure}

\section{Conclusion}
In this technical note, we have addressed
the distributed consensus problem
for a multi-agent system
with general linear dynamics
and a directed leader-follower communication graph.
The main contribution of
this technical note is that
for any communication graph containing a directed spanning
tree with the leader as the root,
a distributed adaptive consensus protocol
is designed, which, depending
on only the agent dynamics and the relative
state information of neighboring agents,
is fully distributed. The case with multiple
leaders has been also discussed.

It is worth mentioning that in this technical note the control input of the leader is assumed to be zero.
A future research direction is to extend the results in this technical note to the general case where
the leader has a bounded control input or the leader is any reference signal
with bounded derivatives.
Another interesting topic for future study
is to extend
the proposed distributed adaptive protocol
to the case of
directed communication graphs without a leader
or to
the case where only relative output information is available.




\begin{thebibliography}{10}

\bibitem{ren2007information}
W.~Ren, R.~Beard, and E.~Atkins, ``{Information consensus in multivehicle
  cooperative control},'' {\em IEEE Control Systems Magazine}, vol.~27, no.~2,
  pp.~71--82, 2007.

\bibitem{hong2008distributed}
Y.~Hong, G.~Chen, and L.~Bushnell, ``{Distributed observers design for
  leader-following control of multi-agent networks},'' {\em Automatica},
  vol.~44, no.~3, pp.~846--850, 2008.

\bibitem{olfati-saber2004consensus}
R.~Olfati-Saber and R.~Murray, ``{Consensus problems in networks of agents with
  switching topology and time-delays},'' {\em IEEE Transactions on Automatic
  Control}, vol.~49, no.~9, pp.~1520--1533, 2004.


\bibitem{antonelli2013interconnected}
G.~Antonelli, ``{Interconnected dynamic systems: An overview on distributed control},''
 {\em IEEE Control Systems Magazine}, vol.~33, no.~1, pp.~76--88, 2013.


\bibitem{Guo2013consensus}
M.~Guo and D.~Dimarogonas, ``Consensus with quantized relative state
  measurements,'' {\em Automatica}, vol.~49, no.~8, pp.~2531--2537, 2013.

\bibitem{li2010consensus}
Z.~Li, Z.~Duan, G.~Chen, and L.~Huang, ``Consensus of multiagent systems and
  synchronization of complex networks: A unified viewpoint,'' {\em IEEE
  Transactions on Circuits and Systems I: Regular Papers}, vol.~57, no.~1,
  pp.~213--224, 2010.

\bibitem{li2011dynamic}
Z.~Li, Z.~Duan, and G.~Chen, ``Dynamic consensus of linear multi-agent
  systems,'' {\em IET Control Theory and Applications}, vol.~5, no.~1,
  pp.~19--28, 2011.

\bibitem{tuna2009conditions}
S.~Tuna, ``Conditions for synchronizability in arrays of coupled linear
  systems,'' {\em IEEE Transactions on Automatic Control}, vol.~54, no.~10,
  pp.~2416--2420, 2009.

\bibitem{seo2009consensus}
J.~Seo, H.~Shim, and J.~Back, ``{Consensus of high-order linear systems using
  dynamic output feedback compensator: Low gain approach},'' {\em Automatica},
  vol.~45, no.~11, pp.~2659--2664, 2009.

\bibitem{zhang2011optimal}
H.~Zhang, F.~Lewis, and A.~Das, ``Optimal design for synchronization of
  cooperative systems: State feedback, observer, and output feedback,'' {\em
  IEEE Transactions on Automatic Control}, vol.~56, no.~8, pp.~1948--1952,
  2011.

\bibitem{ma2010necessary}
C.~Ma and J.~Zhang, ``Necessary and sufficient conditions for consensusability
  of linear multi-sgent systems,'' {\em IEEE Transactions on Automatic
  Control}, vol.~55, no.~5, pp.~1263--1268, 2010.

\bibitem{li2012adaptiveauto}
Z.~Li, W.~Ren, X.~Liu, and L.~Xie, ``{Distributed consensus of linear
  multi-agent systems with adaptive dynamic protocols},'' {\em Automatica},
  vol.~49, no.~7, pp.~1986--1995, 2013.

\bibitem{li2011adaptive}
Z.~Li, W.~Ren, X.~Liu, and M.~Fu, ``{Consensus of multi-agent systems with
  general linear and Lipschitz nonlinear dynamics using distributed adaptive
  protocols},'' {\em IEEE Transactions on Automatic Control}, vol.~58, no.~7,
  pp.~1786--1791, 2013.

\bibitem{su2011adaptive}
H.~Su, G.~Chen, X.~Wang, and Z.~Lin, ``Adaptive second-order consensus of
  networked mobile agents with nonlinear dynamics,'' {\em Automatica}, vol.~47,
  no.~2, pp.~368--375, 2011.

\bibitem{yu2013distributed}
W.~Yu, W.~Ren, W.~X. Zheng, G.~Chen, and J.~L{\"u}, ``Distributed control gains
  design for consensus in multi-agent systems with second-order nonlinear
  dynamics,'' {\em Automatica}, vol.~49, no.~7, pp.~2107--2115, 2013.

\bibitem{sontag1995changing}
E.~Sontag and A.~Teel, ``Changing supply functions in input/state stable
  systems,'' {\em IEEE Transactions on Automatic Control}, vol.~40, no.~8,
  pp.~1476--1478, 1995.

\bibitem{ren2005consensus}
W.~Ren and R.~Beard, ``{Consensus seeking in multiagent systems under
  dynamically changing interaction topologies},'' {\em IEEE Transactions on
  Automatic Control}, vol.~50, no.~5, pp.~655--661, 2005.

\bibitem{bernstein2009matrix}
D.~Bernstein, {\em Matrix Mathematics: Theory, Facts, and Formulas}.
\newblock Princeton University Press, 2009.

\bibitem{qu2009cooperative}
Z.~Qu, {\em {Cooperative Control of Dynamical Systems: Applications to
  Autonomous Vehicles}}.
\newblock London, UK: Springer-Verlag, 2009.


\bibitem{berman1994nonnegative}
A.~Berman and R.~Plemmons, {\em Nonnegative Matrices in the Mathematical
  Sciences}.
\newblock New York, NY: Academic Press, Inc., 1979.

\bibitem{das2010distributed}
A.~Das and F.~Lewis, ``Distributed adaptive control for synchronization of
  unknown nonlinear networked systems,'' {\em Automatica}, vol.~46, no.~12,
  pp.~2014--2021, 2010.

\bibitem{zhang2012lyapunov}
H.~Zhang, F.~Lewis, and Z.~Qu, ``Lyapunov, adaptive, and optimal design
  techniques for cooperative systems on directed communication graphs,'' {\em
  IEEE Transactions on Industrial Electronics}, vol.~59, no.~7, pp.~3026--3041,
  2012.

\bibitem{tuna2008lqr}
S.~Tuna, ``{LQR}-based coupling gain for synchronization of linear
  systems,'' {\em arXiv preprint arXiv:0801.3390}, 2008.

\bibitem{krstic1995nonlinear}
M.~Krsti\'{c}, I.~Kanellakopoulos, and P.~Kokotovic, {\em Nonlinear and
  Adaptive Control Design}.
\newblock New York: John Wiley \& Sons, 1995.

\bibitem{wang2013leader}
C.~Wang, X.~Wang, and H.~Ji, ``Leader-following consensus for an
  integrator-type nonlinear multi-agent systems using distributed adaptive
  protocol,'' in {\em Proceedings of the 10th IEEE International Conference on
  Control and Automation}, pp.~1166--1171, 2013.

\bibitem{cao2012distributed}
Y.~Cao, W.~Ren, and M.~Egerstedt, ``Distributed containment control with
  multiple stationary or dynamic leaders in fixed and switching directed
  networks,'' {\em Automatica}, vol.~48, no.~8, pp.~1586--1597, 2012.

\bibitem{li2013distributed}
Z.~Li, W.~Ren, X.~Liu, and M.~Fu, ``Distributed containment control of
  multi-agent systems with general linear dynamics in the presence of multiple
  leaders,'' {\em International Journal of Robust and Nonlinear Control},
  vol.~23, no.~5, pp.~534--547, 2013.

\bibitem{sturm1999using}
J.~Sturm, ``{Using SeDuMi 1.02, a MATLAB toolbox for optimization over
  symmetric cones},'' {\em Optimization Methods and Software}, vol.~11, no.~1,
  pp.~625--653, 1999.

\end{thebibliography}
\end{document}